\documentclass[12pt]{amsart}
\usepackage{amssymb,latexsym}
\usepackage{amsfonts}
\usepackage{amsmath}
\usepackage{enumerate}
\usepackage{booktabs}
\usepackage[usenames, dvipsnames]{color}
\usepackage[colorlinks,linkcolor=blue,anchorcolor=blue,citecolor=blue]{hyperref}

\newcommand\C{{\mathbb C}}

\newcommand\Q{{\mathbb Q}}
\newcommand\R{{\mathbb R}}
\newcommand\N{{\mathbb N}}
\newcommand\Z{{\mathbb Z}}

\newcommand\gal{\mathrm{Gal}}

\newcommand\Res{\mathrm{Res}}

\newcommand\al{\alpha}
\newcommand\be{\beta}
\newcommand\ga{\gamma}
\newcommand\eps{\varepsilon}

\newtheorem{theorem}{Theorem}[section]
\newtheorem{lemma}[theorem]{Lemma}

\newtheorem{corollary}[theorem]{Corollary}

\theoremstyle{remark}

\numberwithin{equation}{section}

\begin{document}

\title[Conjugates of a Pisot number]{There are no two non-real conjugates of a Pisot number with the same imaginary part}


\author{Art\= uras Dubickas}
\address{Department of Mathematics and Informatics, Vilnius University, Naugarduko 24,
LT-03225 Vilnius, Lithuania}
\email{arturas.dubickas@mif.vu.lt}

\author{Kevin G. Hare}
\address{Department of Pure Mathematics, University of Waterloo, Waterloo, Ontario, Canada N2L 3G1}
\email{kghare@uwaterloo.ca}

\author{Jonas Jankauskas}
\address{Department of Pure Mathematics, University of Waterloo, Waterloo, Ontario, Canada N2L 3G1  and Department of Mathematics and Informatics, Vilnius University, Naugarduko 24,
LT-03225 Vilnius, Lithuania}
\email{jonas.jankauskas@gmail.com}


\thanks{
The research of A. Dubickas and J. Jankauskas was supported in part by the Research Council
of Lithuania Grant MIP-068/2013/LSS-110000-740. The
Research of K. G. Hare was supported in part by NSERC of Canada.}
\thanks{Computational support provided in part by the Canadian Foundation 
        for Innovation, and the Ontario Research Fund.}


\subjclass[2010]{11R06, 11R09}



\keywords{Pisot numbers, Mahler's measure, additive relations}

\begin{abstract}
We show that the number $\al=(1+\sqrt{3+2\sqrt{5}})/2$ with minimal polynomial $x^4-2x^3+x-1$ is the only Pisot number whose four distinct conjugates $\al_1,\al_2,\al_3,\al_4$ satisfy the additive relation $\al_1+\al_2=\al_3+\al_4$. This implies that there exists no two non-real conjugates of a Pisot number with the same
imaginary part and also that at most two conjugates of a Pisot number can have the
same real part. On the other hand, we prove that similar four term equations $\al_1 = \al_2 + \al_3+\al_4$ or $\al_1 + \al_2 + \al_3 + \al_4 =0$ cannot be solved in conjugates of a Pisot number $\al$. We also show that the roots of the Siegel's polynomial $x^3-x-1$ are the only solutions to the three term equation $\al_1+\al_2+\al_3=0$ in conjugates of a Pisot number. Finally, we prove that there exists no Pisot number whose conjugates satisfy the relation $\al_1=\al_2+\al_3$.\end{abstract}

\maketitle



\maketitle

\section{Introduction}\label{intro}

Recall that a
{\it Pisot number} $\alpha>1$ is  an algebraic integer 
whose conjugates over $\Q$, except for $\alpha$ itself, all lie in the
unit disc $|z|<1.$  In 1944, Salem \cite{salem} proved that the set of Pisot numbers is closed, whereas Siegel \cite{sieg} showed that the positive root $\theta=1.32471\dots$ of $x^3-x-1=0$ is the smallest Pisot number, so the number $\theta$ is often called \emph{Siegel's number}. Sometimes it is also called \emph{the plastic number} \cite{pad}. It is worth mentioning that, by the theorem of Smyth \cite{smy0}, Siegel's number has the smallest possible Mahler measure among all non-reciprocal algebraic numbers. For more  information on Pisot and Salem numbers see the book \cite{bert} and more recent surveys \cite{smy3, smy4}.

By a result of Smyth \cite{smy1}, at most two conjugates of a
Pisot number can have the same modulus. Later,
Mignotte \cite{mig} generalized this by proving that there are no non-trivial
multiplicative relations between the conjugates of a Pisot number, namely,

\begin{lemma}[Mignotte]\label{mignot}
The equality $
\al_1^{k_1} \al_2^{k_2} \dots \al_d^{k_d}=1$
with algebraic numbers $\al_1,\al_2,\dots,\al_d$ that are conjugates of a Pisot number $\al$ of degree 
$d$ over $\Q$ and $k_1,k_2,\dots,k_d \in \Z$ can only hold if $k_1=k_2=\dots=k_d$.
\end{lemma} 

The lemma implies, for instance, that no two non-real conjugates of a
Pisot number can have the same argument. Indeed, if the arguments of $\al_1$ and $\al_2$ were equal, then we would
have the non-trivial multiplicative relation $\alpha_1 \overline{\alpha_2} 
\alpha_2^{-1} \overline{\alpha_1}^{-1}=1$ with four conjugates of a Pisot number $\al$ (namely, with $\al_1,\al_2,\overline{\al_1},\overline{\al_2}$; the exponents of other $\deg \al-4$ conjugates in this equality are all equal to zero),
which is impossible, by Lemma~\ref{mignot}.
This simple fact is, basically, everything we know about the geometry of the set of conjugates of a Pisot number.  

In \cite{smydub0},  the first named author and Smyth investigated some non-trivial geometric facts about the set of conjugates of a Salem number. It was proved, for instance, that no three conjugates of a Salem number lie on a line. In conclusion, they asked 
if
two non-real conjugates of a Pisot number can have the same
imaginary part and if 
four conjugates of a Pisot number can have the
same real part. In both cases, the complex (non-real) numbers 
$\al_1,\al_3$ that are conjugates of a Pisot number $\al$
and their complex conjugates $\al_2$, $\al_4$
form a parallelogram (possibly degenerate) in the complex plane $\C$.  Thus,  
a non-trivial additive 
relation 
\begin{equation}\label{m2}
\al_1+\al_2=\al_3+\al_4
\end{equation}
in distinct conjugates $\al_1,\al_2,\al_3,\al_4$ of a Pisot number $\al$ holds.
(Here, $\al_2=\overline{\al_3},\al_4=\overline{\al_1}$ in case
$\al_1,\al_3$ have the same imaginary part, and $\al_2=\overline{\al_1}, \al_4=\overline{\al_3}$ in case $\al_1,\al_3$, where $\al_1 \ne \overline{\al_3}$, have the same real part.)

In \cite{smydub0}, the following example 
\begin{equation}\label{ex_d4}
\alpha:= \al_1 = \frac{1+\sqrt{3+2\sqrt{5}}}{2}
\end{equation}
with minimal polynomial 
\[
f(x)=x^4-2x^3+x-1
\] whose conjugates satisfy \eqref{m2} was found. 
Indeed, the number $\alpha=1.86676\dots$ defined in \eqref{ex_d4} is a Pisot number, with conjugates
$$\alpha_2=\frac{1-\sqrt{3+2\sqrt{5}}}{2}, \>\>\>
\alpha_{3,4}=\frac{1 \pm i\sqrt{-3+2\sqrt{5}}}{2}$$ satisfying
 $\alpha_2=-0.86676\dots,$
$|\alpha_3|=|\alpha_4|=0.78615\dots.$  Hence, two real conjugates
$\alpha_1,$ $\alpha_2$ and two complex conjugate numbers $\alpha_3$
and $\alpha_4=\overline{\al_3}$ satisfy
$\alpha_1+\alpha_2=\alpha_3+\alpha_4=1$, 
so that \eqref{m2} is solvable in conjugates of this Pisot number $\al$. 

Since the property \eqref{m2} is quite unusual, it is quite tempting to conjecture that there are not many (possibly only finitely many) Pisot numbers $\al$ satisfying it. As an unrelated result with Pisot numbers satisfying another odd property one can mention that of Smyth \cite{smy2}, where he proved that there are only eleven special Pisot numbers $\al$ such that $\al/(\al-1)$ is also a Pisot number (this result was motivated by the papers \cite{la1}, \cite{la2}). 

The main theorem of our paper shows that this is indeed the case with the property \eqref{m2}.
More precisely, the special Pisot number \eqref{ex_d4} is a unique Pisot number 
whose conjugates satisfy \eqref{m2}:

\begin{theorem}\label{viens}
If $\al$ is a Pisot number of degree $d \geq 4$ whose four distinct conjugates over $\Q$ satisfy the relation $$\al_1+\al_2=\al_3+\al_4$$ then
\[
\al=\frac{1+\sqrt{3+2\sqrt{5}}}{2}.
\]
Moreover, there exists no Pisot number $\al$ whose four distinct conjugates satisfy the linear relation
\[
\pm \al_1 = \al_2 + \al_3 + \al_4.
\] 
\end{theorem}

Since, by Lemma~\ref{smmm} below, a real and a non-real conjugate 
of any algebraic number cannot have the same real part, Theorem~\ref{viens} implies that, in particular,

\begin{corollary}\label{co12}
No two non-real conjugates of a Pisot number can have the same
imaginary part and at most two conjugates of a Pisot number can have the
same real part. 
\end{corollary}

Note that Corollary~\ref{co12} answers negatively both questions posed in \cite{smydub0}. 
In addition to Theorem \ref{viens}, we will consider a three term linear equations in conjugates of a Pisot number. We show that

\begin{theorem}\label{du1}
If $\al$ is a Pisot number whose three conjugates over $\Q$ satisfy the relation 
\[
\al_1 + \al_2 + \al_3 = 0
\]
then $\al$ is Siegel's number $\theta = 1.32471\dots$ (the 
root of $x^3-x-1=0$). Furthermore, there does not exist a Pisot number $\al$ whose three conjugates
    satisfy the relation
\[
\al_1=\al_2+\al_3.
\]
\end{theorem}

The paper is organized as follows. In Section \ref{aux}, we state some auxiliary results that will be used in our proofs. In Section \ref{du_d7}, we will prove Theorem \ref{du1}. First, we obtain a small upper bound (to be precise $d \leq 8$) for the degree of a potential Pisot number $\al$ whose conjugates might satisfy the linear relations of Theorem~\ref{du1} and restrict the possible candidates to certain short subintervals of $[1, 2]$. Then, we will eliminate all easy cases using various results on the linear relations in conjugate numbers from Section \ref{aux} and some elementary combinatorial arguments. Finally, a small set of remaining candidates will be examined with a computer. A detailed account on our computations is postponed to Section~\ref{calc}. For further investigations of simple three term linear relations among the conjugate algebraic numbers of low degree (not necessarily Pisot numbers) we refer to the forthcoming paper \cite{DJ}.

The proof of the main result, Theorem \ref{viens}, also splits into two parts. The first, diophantine part (bound on the degree $d \leq 18$ and the restrictions on intervals to which all the potential candidates may belong), will be established in Section \ref{proof_viens}. The second, computational part, will be carried out in Section~\ref{calc}. The key difference from the proof of Theorem \ref{du1} is that the computational part is rather heavy and requires many hours of massive calculations distributed on multiple machines.

\section{Auxiliary lemmas}\label{aux}

In addition to the result of Mignotte \cite{mig} that was cited as Lemma \ref{mignot} in Section \ref{intro}, we will need several more results.

Recall that {\it the Weil height}  $h(\ga)$ of an algebraic number $\ga$ of degree $n$ with the minimal polynomial
\[
c_n(x-\ga_1) \dots (x-\ga_n) \in \Z[x], \quad c_n \in \N,
\]
is defined by
\[
h(\ga)=\frac{\log M(\ga)}{n},
\]
where
 \[
 M(\ga)=c_n \prod_{j=1}^n \max\{1,|\ga_j|\}
 \]
 is {\it the Mahler measure} of $\ga$. In particular, if one takes $\ga = \al$, where $\al$ is a Pisot number and $n = \deg \al = d$, then $h(\al) = \log{(\al)}/d$. 
 
We shall use the following result due to Beukers and Zagier, stated as Corollary 2.1 in \cite{beu} :
 
 \begin{lemma}[Beukers and Zagier]\label{sammm}
Let $\be_1,\dots,\be_r$ be non-zero algebraic numbers such that the sum
$N=\be_1+\dots+\be_r$ is a rational integer. If 
\begin{equation}\label{patikr}
\be_1^{-1}+\dots+\be_r^{-1} \ne N
\end{equation}
 then 
$$h(\be_1)+\dots+h(\be_r) \geq \frac{1}{2} \log \left(\frac{1+\sqrt{5}}{2}\right).$$  
\end{lemma}

The result of Beukers and Zagier was motivated by Schinzel's inequality \cite{schi} (see also \cite{schi1}) which asserts that that the Weil height of a totally real algebraic integer $\ga \notin \{-1,0,1\}$  satisfies
\[
h(\ga) \geq \frac{1}{2} \log \left(\frac{1+\sqrt{5}}{2}\right).
\]
 
The generalizations of Lemma \ref{sammm} were proven by Samuels \cite{sam} and Garza, Ishak and Piner \cite{garza}.

We shall also use the next result which was first proved by Kurbatov \cite{kurb}
(see also \cite{dub} for various generalizations and more references on this problem). 

\begin{lemma}[Kurbatov]\label{kurbat}
The equality 
$$
k_1\al_1+k_2 \al_2+ \dots + k_d \al_d=0
$$
with conjugates $\al_1,\al_2,\dots,\al_d$ of an algebraic number $\al$ of prime degree 
$d$ over $\Q$ and $k_1,k_2,\dots,k_d \in \Z$ can only hold if $k_1=k_2=\dots=k_d$.
\end{lemma} 

In \cite{my} Smyth proved that 
\begin{lemma}[Smyth]\label{smmm}
If $\al_1,\al_2,\al_3$ are three conjugates of an algebraic number satisfying 
$\al_1 \ne \al_2$ then $2\al_1 \ne \al_2+\al_3$.  
\end{lemma}

A more general version of Lemma~\ref{smmm} is Theorem 4 of \cite{dub}:
\begin{lemma}[Dubickas]\label{dubickas}
If $\be_1,\be_2, \dots,\be_n$, where $n \geq 3$, are distinct algebraic numbers conjugate over a field $K$ of characteristic zero and $k_1,k_2,\dots,k_n$
are non-zero rational numbers satisfying  $$|k_1| \geq |k_2|+\dots+|k_n|$$
then 
$$k_1 \be_1+k_2 \be_2+\dots+k_n \be_n \notin K.$$
\end{lemma}

We shall also use the following:

\begin{lemma}\label{duu}
A polynomial $f(x)=x^2- x +c$, where $c \ne 1/4$ is a real number, has 
\begin{itemize}
\item[$(i)$] two non-real roots in $|z|<1$, if and only if $1/4<c<1$, 
\item[$(ii)$] two real roots in $(-1,1)$, if and only if $0<c<1/4$,
\item[$(iii)$] a real root in $(-1,1)$ and a real root in $(1,\infty)$, if and only if $-2<c<0$.
\end{itemize}
\end{lemma}

\begin{proof}
Evidently, the roots of $f(x) \in \R[x]$ are complex (non-real) if and only if  
$c>1/4$. Then they are complex conjugate roots and lie in $|z|<1$ if and only if
their product $c$ is less than $1$. Hence, $1/4<c<1$, as claimed in $(i)$. 
Similarly, the roots of $f(x)$ are real (and distinct) if and only if $c<1/4$.
Note that for $c<1/4$, one has $(1+\sqrt{1-4c})/2<1$ if and only if  $c>0$.  One also has
  $(1-\sqrt{1-4c})/2>-1$ in view of $c>-2$. This proves  $(ii)$.  Finally, $f(x)$
has a root in $(-1,1)$ and a root in $(1,\infty)$ whenever 
$f(-1)=2+c>0$ and $f(1)=c<0$. This is equivalent to
$-2<c<0$, as claimed in $(iii)$.  
\end{proof}

The next simple result is based on Kronecker's theorem.

\begin{lemma}\label{duu1}
The polynomials $x,x+1$ and $x^2+x-1$ are the only monic irreducible polynomials in $\Z[x]$ that have all their roots in the interval $(-2,1)$. 
\end{lemma}

\begin{proof} It is clear that there are only two such polynomials $f(x)$ of degree $1$, namely, $x$ and $x+1$. Assume that $\deg{f(x)} \geq 2$.   By Kronecker's theorem (see Theorem 2.5 in \cite{nar}),
$f(x)$ must have a root of the form $2\cos(\pi r)$ with $r \in \Q$. 
Then the largest root of $f(x)$ must be of the form $2\cos (2\pi/D)$ for some $D \in \N$. 
Note that this number is rational for $D=1,2,3,4,6$. Thus, $D \geq 5$ 
in view of $\deg{f(x)} \geq 2$ and the irreducibility of $f(x)$ in $\Z[x]$.  On the other hand, if $D \geq 7$ then $2\cos(2\pi/D) > 1$. Therefore, the only possibility is $D=5$ when the largest root of $f(x)$ is equal to $2\cos(2\pi/5)=2\cos 72^{\circ}=(\sqrt{5}-1)/2.$ 
Consequently, $\deg{f(x)}=2$ and the other root of $f(x)$ must be $(-\sqrt{5}-1)/2$.
Therefore, $$f(x)= (x-(\sqrt{5}-1)/2)(x-(-\sqrt{5}-1)/2)=x^2+x-1,$$ as claimed.
\end{proof}

We conclude Section \ref{aux} with a technical lemma on resultants that will be used in our computer calculations.

\newpage

\begin{lemma}\label{resultant}
Let $f(x) \in \Q[x]$ be the minimal polynomial of an algebraic number $\al$ of degree $d\geq 4$, whose conjugates over $\Q$ are $\al_1$, $\al_2$, $\dots$, $\al_d$. Suppose that $f(x)$ and $f(-x)$ have no common roots. Define the polynomials $g(x) \in \Q[x]$ and $h(x) \in \Q[x]$ by
\[
g(x) := \Res_y\left[f(x-y), f(y)\right], \qquad h(x) := \Res_y\left[f(x+y), f(y)\right]
\] Then, for some four distinct integers $1 \leq i, j, k, l \leq d$,
\begin{enumerate}[i)]
\item the relation $\al_i = \al_j+\al_k$ holds if and only if $f(x)$ divides $g(x)$ and $h(x)$;
\item the relation $\al_i+\al_j+\al_k=0$ holds if and only if $f(-x)$ divides $g(x)$;
\item the relation $\al_i+\al_j=\al_k + \al_l$ holds if and only if $g(x)$ has a factor of multiplicity $\geq 4$  in $\Q[x]$, or, alternatively, $h(x)$ has a factor of multiplicity $ \geq 2$ that is not a power of $x$.
\item the relation $\al_i = \al_j + \al_k + \al_l$ holds if and only if $g(x)$ and $h(x)$ has a common factor in $\Q[x]$.
\item the relation $\al_i + \al_j + \al_k + \al_l = 0$ holds if and only if $g(x)$ and $g(-x)$ has a common factor in $\Q[x]$.
\end{enumerate}
\end{lemma}
\begin{proof} Write $f(x) = \prod_{i=1}^{d} (x-\al_i)$. 
Then the roots of $f(x-y)$ and $f(y)$ with respect to $y$ are $x-\al_i$, $1 \le i \le d$, and $\al_j$, $1 \le j \le d$, 
respectively, so that
\[
g(x) = \Res_y\left[f(x-y), f(y)\right] = (-1)^d\prod_{i=1}^d \prod_{j=1}^d (x-\al_i-\al_j)=
\]
\[
=(-1)^d\prod_{i=1}^{d} (x-2\al_i) \cdot \prod_{1\leq i < j \leq d}(x-\al_i-\al_j)^2 = (-2)^d f(x/2) \cdot u^2(x),
\]
where $u(x) := \prod_{1 \leq i < j \leq d}^d(x-\al_i-\al_j) \in \Q[x]$.

\noindent In a similar way,
\[
h(x) = \Res_y\left[f(x+y), f(y)\right] = \prod_{i=1}^d \prod_{j=1}^d (x+\al_i-\al_j)=
\]
\[
= x^d\cdot \prod_{\substack{i, j = 1\\i \ne j}}^d(x+\al_i-\al_j) = x^d\cdot \prod_{1 \leq i < j \leq d}(x-(\al_j-\al_i))(x+(\al_j-\al_i))= x^d v(x^2),
\]
where $v(x):= \prod_{1 \leq i < j \leq d}^d(x-(\al_i-\al_j)^2) \in \Q[x]$.

Observe that $u(x)$ has no common factors with $f(x/2)$ or $f(-x/2)$; for otherwise, one would have $\pm 2\al_i = \al_j + \al_k$ for $j < k$. Since $f(x)$ has no multiple roots, one has $\al_j \ne \al_k$, so this contradicts Lemmas \ref{smmm} and \ref{dubickas}. In addition to this, $u(x)$ has no factor that is a power of $x$. This is because $\al_i + \al_j=0$, $i < j$, yields $\al_i = -\al_j$ and violates the condition $\gcd\left[f(x), f(-x)\right]=1$ of Lemma \ref{resultant}.

Similarly, $f(x/2)$ has no common factors with $x^d$, since $f(0) \ne 0$, by the condition $\gcd[f(x), f(-x)] = 1$. Also, $f(x/2)$ has no common factors with $v(x^2)$, for otherwise one would have $2\al_k = \al_j - \al_i$ or $\al_k = -\al_i$, in violation of Lemma \ref{smmm}, Lemma \ref{dubickas} or the condition $\gcd\left[f(x), f(-x)\right]=1$. 

For  $(i)$, observe that the equality $\al_i = \al_j + \al_k$ in conjugates of $\al$ is equivalent to the fact that $f(x)$ and $u(x)$ has a common factor in $\Q[x]$. Therefore, $f(x) \mid u(x)$, since $f(x)$ is irreducible in $\Q[x]$. Since $f(x)$ and $2^df(x/2)$ are relatively prime in $\Q[x]$, this is equivalent to $f(x) \mid g(x)$. Alternatively, one can rewrite $\al_i = \al_j + \al_k$ as $\al_k = \al_i - \al_j$, and notice that this equation is equivalent to the fact that $f(x) \mid h(x)$.

In $(ii)$, $-\al_i = \al_j + \al_k$ is equivalent to $f(-x) \mid g(x)$.

In $(iii)$, the relation $\al_i + \al_j = \al_k + \al_l$, where $\{i, j\} \ne \{k, l\}$ is equivalent to the fact that $u(x)$ has a square factor $(x-\al_i-\al_j)^2$. Since $f(x/2)$ has no multiple roots, this is equivalent to the fact that $g(x)$ contains factors of multiplicity $\geq 4$ in $\Q[x]$. Alternatively, looking into this relation as $\al_i - \al_k = \al_l - \al_j$, one sees that it is equivalent to $h(x)$ being divisible by a square factor in $\Q[x]$ that is not equal to $x^2$.

In $(iv)$, the equation $\al_i -\al_j = \al_k + \al_l$ is equivalent to $u(x)$ and $v(x^2)$ having a common factor in $\Q[x]$. By co-primality of $f(x/2)$ and $x^d$, $f(x/2)$ and $v(x^2)$, $u(x)$ and $x^d$ established earlier in the proof, this is equivalent to the fact that $g(x)$ and $h(x)$ has a common factor in $\Q[x]$. 

Finally, for part $(v)$, $\al_i + \al_j = -(\al_k + \al_l)$ is equivalent to $u(x)$ and $u(-x)$ having a common factor in $\Q[x]$. By mutual co-primality of $u(x)$, $f(x/2)$, $f(-x/2)$, this is equivalent to $g(x)$ and $g(-x)$ having a common factor in $\Q[x]$.
\end{proof}
\section{Proof of Theorem~\ref{du1}}\label{du_d7}

Assume that there exists a Pisot number $\al$ whose conjugates 
satisfy 
\begin{equation}\label{lin3}
\al_1=\al_2+\al_3 \quad \text{ or } \quad \al_1 + \al_2 + \al_3 = 0.
\end{equation} 
Clearly, in both equations all three conjugates $\al_1, \al_2, \al_3$ must be distinct, so we have $d \geq 3$.

First, we will show that $\al < 2$. For this, let $\sigma$ be the Galois automorphism that maps $\al_1$ into the Pisot number $\al$. From \eqref{lin3}, one obtains
\[
\al = \sigma(\al_2) + \sigma(\al_3) \quad \text{ or } \quad \al = -\sigma(\al_2)-\sigma(\al_3).
\] Since $\sigma$ is a bijection, both $\sigma(\al_2)$ and $\sigma(\al_3)$ must be $ \leq 1$ in absolute value. Hence, in both cases, 
\[
\al=|\al| \leq |\sigma(\al_2)|+|\sigma(\al_3)|<1+1=2.
\] 
The second step is to obtain an upper bound $d= \deg \al \leq 8$.

Assume that the first relation in \eqref{lin3} holds.  We apply Lemma~\ref{sammm} to
\[
\be_1:=\al_1, \quad \be_2:=-\al_2, \quad \be_3:=-\al_3 \text{ and } N:=0.
\]
One needs to verify the condition \eqref{patikr}. Indeed, if 
\[
\be_1^{-1}+\be_2^{-1}+\be_3^{-1}= \al_1^{-1}-\al_2^{-1}-\al_3^{-1}=0.
\]
Then one has $\al_1^{-1} = \al_2^{-1}+\al_3^{-1}$, or $\al_2\al_3 = \al_1 (\al_2 + \al_3)$. On the other hand, we have $\al_1 = \al_2 + \al_3$. This yields
\[
\al_2 \al_3=\al_1(\al_2+\al_3)=\al_1^2.
\]
However, the non-trivial multiplicative relation $\al_1^2 = \al_2 \al_3$ contradicts Lemma~\ref{mignot}. Hence, the condition of Lemma \ref{sammm} is satisfied. Now, by Lemma~\ref{sammm}, we obtain
\begin{align*}
\frac{3\log \al}{d} &=3h(\al)  =h(\al_1)+h(-\al_2)+h(-\al_3) = \\&=h(\be_1)+h(\be_2)+h(\be_3) 
 \geq \frac{1}{2} \log \left(\frac{1+\sqrt{5}}{2}\right),
\end{align*}
since $h(\ga)=h(-\ga)$. Therefore,
\begin{equation}\label{bound1}
\frac{3\log{\al}}{d} \geq \frac{1}{2}\log{\left(\frac{1+\sqrt{5}}{2}\right)}.
\end{equation} By combining the inequality \eqref{bound1} with the estimate $1<\al < 2$, one obtains
\begin{equation}\label{bound_d8}
d \le \frac{6 \log 2}{\log \left(\frac{1+\sqrt{5}}{2}\right)}=8.64252\dots.
\end{equation} Consequently, $d$ can only take the values $3,4,5,6,7,8$.

The last step is to rewrite inequality \eqref{bound1} as
\begin{equation}\label{interv3}
\left(\frac{1+\sqrt{5}}{2}\right)^{d/6} \leq \al < 2, \qquad 3 \leq d \leq 8.
\end{equation}

Inequality \eqref{interv3} also holds for any Pisot number $\al$ whose conjugates satisfy the second linear equation in \eqref{lin3} instead of the first one. This can be proved by the application of Lemma \ref{sammm} to the linear relation $\al_1+\al_2+\al_3=0$ exactly in the same manner as we did it for the equation $\al_1=\al_2+\al_3$, by setting $\be_j:=\al_j$ for $j=1,2,3$ and then deriving \eqref{bound1}, \eqref{bound_d8} and \eqref{interv3}.  


We further refine the list of candidates. First we consider the possible solutions to the linear equation $\al_1 + \al_2 + \al_3 = 0$. We will show that $d \in \{3, 6\}$. Indeed, in prime degree cases $d \in \{5, 7\}$, the linear relation $\al_1+\al_2+\al_3=0$ contradicts Lemma \ref{kurbat}. Now, let $t \in \Q$ be the sum of all distinct conjugates of $\al$ over $\Q$. In case $d=4$, the trace formula  $t=\al_1+\al_2+\al_3+\al_4$ together with $\al_1+\al_2+\al_3 = 0$ implies $\al_4 \in \Q$, which is impossible. 

Assume next that $d=8$. Put $t$ for the trace of $\al$. Let $N$ be  the number of distinct equalities
$\al_i+\al_j+\al_k=0$ obtained by applying all automorphisms 
of the Galois group $\text{Gal}(\Q(\al)/\Q)$ to $\al_1+\al_2+\al_3=0$. 
Then each $\al_i$ occurs exactly $\ell$ times in these equalities, so that 
$3N=8\ell$. In particular, $\ell$ must be divisible by $3$, so $\ell \geq 3$. Note that the intersection of two distinct sets of indices $\{i,j,k\}$ and $\{i',j',k'\}$ satisfying
$\al_i+\al_j+\al_k=0$ and $\al_{i'}+\al_{j'}+\al_{k'}=0$ is either empty or contains at most one element. So, by considering the equalities of the form
$\al_1+\al_i+\al_j=0$ (there are at
 least three such equalities), we find that after re-indexing the conjugates
the following three equalities
 $$\al_1+\al_2+\al_3=0, \>\> \al_1+\al_4+\al_5=0,\>\> \al_1+\al_6+\al_7=0$$ hold. By adding these equalities, we obtain  
 \[
 2\al_1+\al_1+\al_2+\dots+\al_7=2\al_1+t-\al_8=0,
 \] so that $2\al_1 - \al_8 = k \in \Q$. However, this contradicts to Lemma~\ref{dubickas}. Therefore, $d \in \{3, 6\}$. The minimal polynomials of Pisot numbers $\al$ of degree $3$ and degree $6$ satisfying \eqref{interv3} are listed in Tables \ref{pisot_d3} and \ref{pisot_d6}, respectively. In Section \ref{calc}, we will give a detailed account on how these polynomials were calculated. The number $\tau$ stands for the golden section number $(1+\sqrt{5})/2$.  

\begin{table}\caption{All Pisot numbers $\al \in \left(\tau^{1/2}, 2 \right)$ of degree $d=3$}
\begin{tabular}{l l l}
\toprule
no.& $\al$ & Minimal polynomial of $\al$\\
\midrule
1. & $1.32471\dots$ & $x^3 - x - 1$\\
2. & $1.46557\dots$ & $x^3 - x^2 - 1$\\
3. & $1.75487\dots$ & $x^3 - 2x^2 + x - 1$\\
4. & $1.83928\dots$ & $x^3 - x^2 - x - 1$\\
\bottomrule
\end{tabular}\label{pisot_d3}
\end{table}

\begin{table}\caption{All Pisot numbers $\al \in \left(\tau, 2 \right)$ of degree $d=6$}
\begin{tabular}{l l l}
\toprule
no. & $\al$ & Minimal polynomial of $\al$\\
\midrule
1. & $1.71428\dots$ & $x^6 - x^5 - x^4 - x^2 + 1$\\
2. & $1.80750\dots$ & $x^6 - x^5 - x^4 - x^3 + 1$\\
3. & $1.91616\dots$ & $x^6 - 2x^5 + x^4 - 2x^3 + x^2 - x + 1$\\
4. & $1.98138\dots$ & $x^6 - x^5 - x^4 - 2x^3 + 1$\\
5. & $1.66040\dots$ & $x^6 - x^5 - x^3 - x^2 - 1$\\
6. & $1.96716\dots$ & $x^6 - 2x^5 + x - 1$\\
7. & $1.74370\dots$ & $x^6 - x^5 - x^4 - x - 1$\\
8. & $1.78711\dots$ & $x^6 - x^4 - 2x^3 - 2x^2 - 2x - 1$\\
9. & $1.80509\dots$ & $x^6 - 2x^5 + x^4 - x^3 - 1$\\
10. & $1.89382\dots$ & $x^6 - 2x^5 + x^2 - 1$\\
11. & $1.91118\dots$ & $x^6 - x^5 - x^4 - x^3 - x - 1$\\
12. & $1.95545\dots$ & $x^6 - 3x^5 + 3x^4 - 2x^3 + x - 1$\\
13.  & $1.97947\dots$ & $x^6 - x^5 - 2x^4 + x^2 - x - 1$\\
14. & $1.98358\dots$ & $x^6 - x^5 - x^4 - x^3 - x^2 - x - 1$\\
\bottomrule
\end{tabular}\label{pisot_d6}
\end{table}
By searching Table \ref{pisot_d3} for polynomials with the trace $0$, we see that only the roots $\{\al_1, \al_2, \al_3\}$ of the polynomial $x^3-x-1$ satisfy the equation $\al_1 + \al_2 + \al_3 = 0$. A quick computation with \texttt{Maple} computer algebra package revealed that no polynomial $f(x)$ of degree $6$ in Table \ref{pisot_d6} satisfies the condition in part $(ii)$ of Lemma \ref{resultant}. Hence, there are no more solutions to the equation $\al_1+\al_2+\al_3 = 0$. An alternative way to check the polynomials of degree $6$ in Table \ref{pisot_d6} without computer is to apply (a more general) Theorem 1.2 from \cite{DJ}.

Now let us consider the linear relation $\al_1 = \al_2 + \al_3$ in conjugates of a Pisot number $\al$. Note that, by Lemma~\ref{kurbat}, $d \ne 3,5,7$, so it remains to consider the cases $d=4,6,8$. 
Consider an automorphism $\sigma$ of the Galois group $\gal(\Q(\al)/\Q)$
that maps $\al_1$ to $\al_2$. Setting $\sigma(\al_2)=\al_k$ and $\sigma(\al_3)=\al_l$, we obtain $\al_1=\al_2+\al_3$ and $\al_2=\al_k+\al_l$. We claim that $k,l>3$. Indeed, otherwise, as 
$\al_k$ and $\al_l$ are distinct,
$k \ne l$, assuming without loss of generality that $k<l$, we must have $1 \leq k \leq 3$. Clearly, $k \ne 2$. If $k=1$ then, by adding both equalities, we obtain
$\al_3+\al_l=0$, which is impossible 
(otherwise, $-\al$ is a conjugate of $\al$, so $\al$ is not a Pisot number). If $k=3$ then, by subtracting $\al_2=\al_3+\al_l$ from $\al_1=\al_2+\al_3$, we find that  
$2\al_2=\al_1+\al_l$, which contradicts Lemma~\ref{smmm}.  Thus, without restriction of generality, we may assume that $k=4$ and $l=5$, namely,
\begin{equation}\label{dvilyg}
\al_1=\al_2+\al_3 \>\>\> \text{and} \>\>\> \al_2=\al_4+\al_5.
\end{equation}
 In particular, this shows that
$d \geq 5$, so $d$ cannot be $4$.  

Assume next that $d=6$ and put $t$ for the sum of all $6$ conjugates of $\al$. Then, by adding both equalities in \eqref{dvilyg}, we find that
$$\al_1+\al_2=\al_2+\al_3+\al_4+\al_5=t-\al_1-\al_6.$$  Thus, $2\al_1+\al_2+\al_6=t \in \Z$,
which contradicts Lemma~\ref{dubickas} with $n=3$ and $K=\Q$. 

Hence, it remains to solve the equation $\al_1=\al_2+\al_3$ in conjugates of Pisot numbers of degree $d = 8$ that, by inequalities \eqref{interv3}, are restricted to the interval $(\tau^{4/3}, 2)$. All candidates are listed in Table \ref{pisot_d8} (see also Section \ref{calc}). For each of the numbers $\al$ given in Table \ref{pisot_d8}, we have checked the conditions of Lemma \ref{resultant} $(i)$ with Maple and verified that $\al_1 \ne \al_2 + \al_3$. Therefore, the equation $\al_1 = \al_2+\al_3$ is not solvable in conjugates of a Pisot number. This completes the proof of Theorem \ref{du1}.

\begin{table}\caption{All Pisot numbers $\al \in \left(\tau^{4/3}, 2 \right)$ of degree $d=8$.}
\begin{tabular}{l l l}
\toprule
no. & $\al$ & Minimal polynomial\\
\midrule
1. & $1.94284\dots$ & $x^8 - 2x^7 + x^4 - x^3 - x + 1$\\
2. & $1.96113\dots$ & $x^8 - x^7 - x^6 - x^5 - x^4 - x^3 + 1$\\
3. & $1.92172\dots$ & $x^8 - x^7 - x^6 - x^5 - x^4 + 1$\\
4. & $1.94653\dots$ & $x^8 - 2x^7 + x^5 - 2x^4 + x^3 - x + 1$\\
5. & $1.99577\dots$ & $x^8 - x^7 - x^6 - x^5 - 2x^4 + 1$\\
6. & $1.92600\dots$ & $x^8 - x^7 - 2x^6 + x^4 - x^2 + 1$\\
7. & $1.99203\dots$ & $x^8 - 2x^7 + x - 1$\\
8. & $1.97061\dots$ & $x^8 - x^7 - x^6 - x^5 - x^4 - x^3 - 1$\\
9. & $1.91743\dots$ & $x^8 - 3x^7 + 3x^6 - 2x^5 + 2x^3 - 3x^2 + 2x - 1$\\
10. & $1.90988\dots$ & $x^8 - x^7 - x^6 - x^5 - x^4 + x + 1$\\
11. & $1.91580\dots$ & $x^8 - 2x^7 + x^5 - x^4 - x^3 + x^2 - 1$\\
12. & $1.96225\dots$ & $x^8 - x^7 - 2x^6 + x^4 - x^3 - x^2 + x + 1$\\
13. & $1.97526\dots$ & $x^8 - 2x^7 + x^2 - 1$\\
14. & $1.99402\dots$ & $x^8 - x^7 - 2x^6 - x^5 + x^4 + 2x^3 + x^2 - x - 1$\\
15. & $1.93167\dots$ & $x^8 - 2x^7 - x^6 + 3x^5 - x^4 - 2x^3 + 2x^2 - 1$\\
16. & $1.91451\dots$ & $x^8 - x^7 - 2x^6 - x^5 + 2x^4 + 2x^3 - x - 1$\\
17. & $1.93895\dots$ & $x^8 - 2x^7 + x^3 - 1$\\
18. & $1.95731\dots$ & $x^8 - 2x^6 - 3x^5 - 2x^4 + 2x^2 + 2x + 1$\\
19. & $1.98707\dots$ & $x^8 - 3x^7 + 2x^6 + x^5 - 2x^4 + x - 1$\\
20. & $1.99603\dots$ & $x^8 - x^7 - x^6 - x^5 - x^4 - x^3 - x^2 - x - 1$\\
\bottomrule
\end{tabular}\label{pisot_d8}
\end{table}

\smallskip

\noindent \emph{Remark 1.} Combinatorial arguments may be used to prove that the equation $\al_1 \ne \al_2 + \al_3$ has no solutions in conjugate algebraic numbers of degree $8$.  However, the proof is much longer and it is not practical to give it here. For details, refer to \cite{DJ} again.
\smallskip

\noindent \emph{Remark 2 .} Another way to prove the non-existence of solutions of the equation $\al_1 = \al_2 + \al_3$ when $d=8$  in Theorem \ref{du1} is to use  Theorem~\ref{viens}. Indeed, by mapping $\al_1$ to $\al_j$
for $j=2,3,\dots,8$, we obtain eight equalities of the form
$$\al_j=\al_{k(j)}+\al_{l(j)}$$ with $j=1,2,\dots,8$ and $k(j)< l(j)$ lying in the set $\{1,2,\dots,8\} \setminus \{j\}$. At least one of these indices, say,
$k(1)=2$ appears at least twice on the right hand side of these eight equalities, namely, $\al_1=\al_2+\al_3$ and $\al_j=\al_2+\al_l$, where $j \ne 1$ and so $l \ne 3$. Subtracting the first equality from the second we find that $\al_1-\al_j=\al_3-\al_l$. Thus, $\al_1+\al_l=\al_3+\al_j$. By Theorem~\ref{viens}, this is only possible when $d=4$, contrary to $d=8$.

We must stress, however, that from the computational point of view, there is not much sense in proving Theorem \ref{du1} through Theorem \ref{viens}, since the direct search on the small list of candidates in Table \ref{pisot_d8} is much simpler than the large amount of computations required in the proof of Theorem \ref{viens}.    

\section{Beginning of the proof of Theorem~\ref{viens}}\label{proof_viens}

Let $\mu, \nu \in \{-1, 1\}$. Assume that there is a Pisot number $\al$ whose four distinct conjugates 
$\al_1,\al_2,\al_3,\al_4$ satisfy the additive relation
\[
\mu\al_1+\nu\al_2+\al_3+\al_4=0.
\] The choice of parameters $\mu=\nu = -1$ yields the first equation of the theorem $\al_1 + \al_2 = \al_3 + \al_4$. The choices $\mu = -1, \nu = 1$ and $\mu = \nu = 1$ correspond to the equations $\al_1 = \al_2 + \al_3 +\al_4$ and $\al_1 + \al_2 + \al_3 + \al_4 = 0$, respectively.

By applying an automorphism 
$\sigma$ of the Galois group $\gal(\Q(\al)/\Q)$ that sends $\al_1$ to $\al$,
we deduce
$\mu\al=-\nu\sigma(\al_2)-\sigma(\al_3)-\sigma(\al_4)$. Since $|\sigma(\al_j)|<1$ for $j=2,3,4$, this 
yields $\al<3$. 

 This time we will apply Lemma~\ref{sammm} to
 \[
 \be_1:=\mu\al_1, \quad \be_2:=\nu\al_2, \quad \be_3:=\al_3, \quad \be_4:=\al_4 \text{ and } N:=0.
 \] One needs to verify the condition \eqref{patikr}. Observe that the equation 
\[
\be_1^{-1}+\be_2^{-1}+\be_3^{-1}+\be_4^{-1}= \mu\al_1^{-1}+\nu\al_2^{-1}+\al_3^{-1}+\al_4^{-1}=0,
\]
yields
\[
(\mu\al_1+\nu\al_2)\al_3\al_4=-\mu\nu(\al_3+\al_4)\al_1\al_2.
\] 
Clearly, $\al_3+\al_4 \ne 0$, since otherwise $-\al$ is a conjugate of $\al$, which is impossible. Thus, dividing both sides by $-(\al_3+\al_4)=\mu\al_1+\nu\al_2 \ne 0$ we obtain the multiplicative 
relation $\al_3 \al_4=\mu\nu \al_1 \al_2$. Squaring both sides yields $\al_1^2\al_2^2 =\al_3^2\al_4^2$. In view of Lemma~\ref{mignot}, such an identity cannot hold, since the numbers $\al_1^2$, $\al_2^2$, $\al_3^2$ and $\al_4^2$ are distinct conjugates of a Pisot number $\al^2$. 
 Therefore, by Lemma~\ref{sammm}, we 
obtain
\[
4h(\al)=h(\mu\al_1)+h(\nu\al_2)+h(\al_3)+h(\al_4) \geq \frac{1}{2} \log \left(\frac{1+\sqrt{5}}{2}\right).
\]
In view of $h(\pm\al)=\log{M(\al)}/d$ and $M(\al)=\al$, one can rewrite the last inequality as
\begin{equation}\label{bound2}
\frac{4\log{\al}}{d} \geq \frac{1}{2}\log{\left(\frac{1+\sqrt{5}}{2}\right)}.
\end{equation}
Since $1 < \al < 3$, the inequality \eqref{bound2} yields
\begin{equation}\label{d18}
d \le \frac{8 \log 3}{\log \left(\frac{1+\sqrt{5}}{2}\right)}=18.26409\dots.
\end{equation}

Consequently, the degree $d$ of $\al$ can only take the values in the range 
$4 \leq d \leq 18$. We can also rewrite inequalities \eqref{bound2}, \eqref{d18} as
\begin{equation}\label{interv4}
\left(\frac{1+\sqrt{5}}{2}\right)^{d/8} \le  \al < 3, \quad 4 \leq d \leq 18.
\end{equation}

In the remainder of this section we will show that in case $d=4$
the only Pisot number $\al$ whose conjugates satisfy \eqref{m2} 
is precisely the number $(1+\sqrt{3+2\sqrt{5}})/2$. 

Assume that $\al$ is a Pisot number of degree $4$ with conjugates 
$\al=\al_1,\al_2,\al_3,\al_4$ and trace $t=\al_1+\al_2+\al_3+\al_4 \in \Z$. 
From $\al_1+\al_2=\al_3+\al_4$, we see that $\al_1+\al_2=t/2$. Next, from
$\al_1+\al_2=\al+\al_2>0$ and $|\al_1+\al_2|=|\al_3+\al_4|<1+1=2$ we obtain $0<t/2<2$. Hence, $t \in \{1,2,3\}$. Furthermore, $t/2=\al_1+\al_2$
is the sum of two algebraic integers, so $t/2$ is an algebraic integer. 
Thus, the only choice for $t$ is $t=2$. Consequently, 
$\al_1+\al_2=\al_3+\al_4=1$. 

It follows that the minimal polynomial $f(x) \in \Z[x]$ of $\al$ is
\begin{align*}
f(x) &=(x-\al_1)(x-\al_2)(x-\al_3)(x-\al_4) =\\
&= (x^2-x+\al_1\al_2)(x^2-x+\al_3\al_4)=\\
&=(x^2-x+\be_1)(x^2-x+\be_2)=\\
&=(y+\be_1)(y+\be_2) =: g(y),
\end{align*}
where $y:=x^2-x$. Therefore, $\be_1=\al_1\al_2=\al_1(1-\al_1)<0$ and $\be_2=\al_3\al_4$ are real quadratic algebraic integers that are conjugate over $\Q$ (since $f(x)$ is irreducible). Since $\al_1 > 1$, $\al_2 = 1-\al_1$ and from the fact that $\al_1$ is a Pisot number, it follows that $\al_2 \in (-1, 0)$. From Lemma~\ref{duu} $(iii)$, it follows that
$\be_1 \in (-2,0)$. Similarly, as $|\al_3|<1$ and $|\al_4|<1$,  Lemma~\ref{duu} $(i)$ and $(ii)$
 implies that $\be_2$ must be in the interval $(0,1)$. Thus, by Lemma~\ref{duu1}, the quadratic polynomial
$g(-y)=(y-\be_1)(y-\be_2) \in \Z[y]$ must be $y^2+y-1$, which gives
\[
\be_1 = \frac{-1-\sqrt{5}}{2}, \quad \be_2 = \frac{-1+\sqrt{5}}{2}.
\]
Hence, $g(y)=y^2-y-1$. By inserting $y=x^2-x$, we find that 
\[
f(x) =(x^2-x-\be_1)(x^2-x-\be_1)=g(x^2-x)=x^4-2x^3+x-1
\]
with the root $\al=(1+\sqrt{3+2\sqrt{5}})/2$, as claimed. Clearly, the 
roots $\al_1,\al_2,\al_2,\al_4$ of $f(x)$ satisfy
$\al_1+\al_2=\al_3+\al_4=1$.

It remains to show that in cases $5 \le d \le 18$, no four distinct conjugates of a Pisot number of degree $d$ satisfy \eqref{m2}, and there exists no Pisot number $\al$ whose conjugates satisfy $\al_1 = \al_2 + \al_3 + \al_4$ or   $\al_1 + \al_2 + \al_3 + \al_4 = 0$ for $4 \leq d \leq 18$.

\section{Calculations and the end of the proof of Theorem \ref{viens}}\label{calc}

Computationally, the problem is primarily one of finding the minimal polynomials of all Pisot numbers in subintervals of the interval $[1,3]$ of the appropriate degree, and then testing the equations \eqref{m2} and \eqref{lin3}.

To find all Pisot numbers in a given interval up to a given degree, we wrote a fast implementation of the Boyd's algorithm \cite{Boyd78, Boyd84, Boyd85} in the {\tt C} language. Our implementation is based on {\tt FLINT} ({\bf F}ast {\bf Li}brary for {\bf N}umber {\bf T}heory) version 2.4.4 \cite{flint}. The {\tt FLINT} library provides the implementations of the polynomials in $\Z[x]$ and $\Q[x]$ with coefficients of arbitrarily large size and is highly optimized for fast arithmetical operations. 

The initial searches for small degrees (up to $d \leq 17$) were done on a single RedHat Linux  server equipped with two Intel Xeon X5672 series 3.20GHz 12MB Cache 1333MHz 95W CPUs and 96735MB of RAM that was running at the University of Waterloo computing facilities.

First we ran the search to find all the Pisot numbers up to degree $d \leq 8$ in the interval $[1, 2]$. The program found $109$ such Pisot numbers. We processed this list on {\tt Maple} and removed the numbers $\al \not\in (\tau^{d/6}, 2)$ that do not satisfy the inequalities \eqref{interv3}, leaving only $78$ Pisot numbers. Totals from the final list are recorded in 
Table \ref{countss1}. 

\begin{table}[h]\caption{\# of Pisot numbers $\al \in (\tau^{d/6}, 2)$, for $3 \leq d \leq 8$}
\begin{tabular}{ccccccc}
\toprule
$\deg \al$		&	$3$	&	$4$	&	$5$	&	$6$	&	$7$	&	$8$	\\
\midrule
\# of $\al$'s	&	$4$	&	$4$	&	$12$	&	$14$	&	$24$	&	$20$	\\
\bottomrule	
\end{tabular}\label{countss1}
\end{table}

The Pisot numbers for degrees $d=3, 6$ and $9$ satisfying $\alpha \in (\tau^{d/6},2)$ are given in Tables \ref{pisot_d3} \ref{pisot_d6} and \ref{pisot_d8}.
These were manually checked in Section \ref{du_d7} to prove the remaining cases of Theorem \ref{du1}.

To complete the proof of Theorem \ref{viens}, one needs to find all Pisot numbers up to degree $18$ in the interval $(1, 3)$ satisfying \eqref{interv4}. The searches on a single computer were feasible for $d \leq 17$, but for $d=18$, this was no longer practical: incrementing the degree $d$ by $1$ on the same interval resulted in multiplying the search time by a factor in the range $4$ to $4.4$, and doubling the memory usage. The predicted the search time on a single machine would be  up to two CPU months while the RAM usage was predicted to remain under $1$ GB. The single machine search timings are recorded in Table \ref{singleCPU}.

\begin{table}[h]\caption{Single Intel Xeon 3.4GHz machine search timings}
\begin{tabular}{lll}
\toprule
$ \deg{\al}$ & Search interval	& CPU time			\\
\midrule
$10$	&	$[1, 3]$				& $25$ sec.		\\
$11$	&	$[1, 3]$				& $2$ \ min. $13$ sec.\\
$12$	&	$[1, 3]$				& $11$ min. $7$	\\
$13$	&	$[1, 3]$				& $54$ min.	  	\\
$14$	&	$[1, 3]$				& $4$ \ h. $13$ min.	\\
$15$	&	$[1, 3]$				& $18$ h. $47$ min.	\\
$16$	&	$\left[ 610/233, 3\right]$	& $3$ \ days $11$ h.  	\\
$17$	&	$\left[ 367/132, 3\right]$	& $13$ days $17$ h. \\
$18$	&	$\left[ 437/148, 3\right]$	& \textcolor{Gray}{$\leq 60$ days (estimated)}\\
\bottomrule
\end{tabular}\label{singleCPU}
\end{table}

Consequently, we decided to distribute the computations on a large collection of 2 Intel 5272 series 3.4Ghz/6M/1600Mhz 80W Dual Core Xeon Processor machines, allowing up to 120 simultaneous searches to be done. This was achieved by partitioning the search interval $[1,3]$ into 2868 subintervals.
The lengths of these subintervals were balanced with the degree of the 
    polynomials being searched.  
By inequality \eqref{interv4}, the closer a Pisot number $\al$ is to $x=3$ on the real line, the larger are the degrees that must be search for, and in turn, the smaller the subintervals that are searched.
This is necessary because for any fixed interval there are considerably more Pisot numbers of large degree in this interval than Pisot numbers of smaller degrees.
For example, our search found $40\,875$ Pisot numbers of degree 12 in the interval $[2, 3]$, while a much shorter interval between the points $ 1126/405 \approx 2.7802$ and $3$ contains $630\,165$ Pisot numbers of degree $17$.
In practice, the intervals up to degree $d = 12$ were of fixed length $\eps=1/10$. For $d \geq 12$, we used the intervals of length $\eps(d) = 3^{12-d}\cdot 10^{-1}$.

Distributed computations took 1$3.64$ CPU days. In total, $1\,956\,289$ Pisot numbers were found. It should be noted that the number of Pisot polynomials found is somewhat 
higher than actual number of Pisot numbers satisfying the inequality \eqref{interv4}.
This is because we searched over a slightly larger collection of intervals $[a, b]$ with rational endpoints $a, b$ that cover all intervals $(\tau^{d/8}, 3)$, for $4 \leq d \leq 18$. 

After the minimal polynomials were computed, they were sieved by checking if their Pisot roots lie in intervals restricted by the inequalities \eqref{interv4}. All $1\,955\,183$ such Pisot numbers are counted in Table \ref{count222}. 

\begin{table}[h]\caption{\# of Pisot numbers $\al \in (\tau^{\deg{(\al)}/8}, 3)$}
\begin{tabular}{llllll}
\toprule
$\deg{\al}$ & \# of $\al$'s & $\deg{\al}$ & \# of $\al$'s & $\deg{\al}$ & \# of $\al$'s \\
\midrule
$4$ & $43$           & $9$   & $5\,555$  & $14$   & $140\,587$\\
$5$ & $162$         & $10$ & $9\,937$  & $15$ & $273\,851$\\
$6$ & $353$         & $11$ & $23\,410$ & $16$ & $402\,209$\\
$7$ & $1\,075$     & $12$ & $40\,812$ & $17$ & $630\,025$\\
$8$ &  $2\,069$    & $13$ & $85\,979$ & $18$ & $339\,116$\\
\bottomrule   
\end{tabular}\label{count222}
\end{table}
 
 By the result of Kurbatov stated as Lemma \ref{kurbat} in Section \ref{aux}, $d \not\in \{5,7,11,13,17\}$. So the next step of the sieve was to remove those polynomials of prime degree. This reduced the number of eligible Pisot polynomials to $1\,214\,532$.
The third step was to find the minimal polynomials of Pisot numbers whose roots satisfy one of the three numerical inequalities
\[
|\alpha_1 - \alpha_2 - \alpha_3 - \alpha_4 |  < 10^{-5}, \qquad |\alpha_1 + \alpha_2 + \alpha_3 + \alpha_4 |  < 10^{-5},
\] or
\[
|\alpha_1 + \alpha_2 - \alpha_3 - \alpha_4 |  < 10^{-5}.
\]
These calculations were done with 10 digits of accuracy, so this would get all potential Pisot numbers satisfying equation the equations of Theorem \ref{viens}. Some of these \emph{false positive} examples are shown in Table \ref{tab:fake}. 

\begin{table}[h]\caption{False positive examples, where $\alpha_1 + \alpha_2 \approx \alpha_2 + \alpha_4$}
\begin{tabular}{p{3in}c}
\toprule
\ \ \ \ \ \ \ \ \ \ \ \ Pisot polynomial $f(x)$ & $|\alpha_1 + \alpha_2 - \alpha_3 - \alpha_4|$\\
\midrule
$x^{15}-3 x^{14}+x^{13}+x^{12}-2 x^{11}+2 x^{10}-2 x^9+x^8+x^7-2 x^6+2 x^5-2 x^4+x^3+x^2-2 x+1 $ & $ 0.61690 \times 10^{-8} $\\ \\
$x^{15}-2 x^{14}-2 x^{13}-x^{12}-3 x^{11}-3 x^{10}-2 x^8-2 x^7-x^6-x^5-2 x^4-x^3-x^2-x-1 $ & $  0.16262 \times 10^{-7} $\\ \\
$x^{18}-x^{17}-3 x^{16}-5 x^{15}-7 x^{14}-8 x^{13}-7 x^{12}-6 x^{11}-4 x^{10}-2 x^9+x^7+x^6+x^5-x^3-x^2-x-1 $ & $ 0.34922 \times 10^{-7} $\\ \\
$x^{18}-2 x^{17}-2 x^{16}-2 x^{15}-2 x^{14}-x^{13}-2 x^{12}-2 x^{11}-x^{10}-x^9+x^8+x^7+x^6+x^5+x^4+2 x^3+2 x^2+x+1 $ & $ 0.18618 \times 10^{-6} $\\ \\
$x^{16}-3 x^{15}+x^{14}-3 x^{13}+2 x^{12}-2 x^{11}+x^{10}-2 x^9-2 x^7+x^6-2 x^5+2 x^4-3 x^3+x^2-2 x+1 $ & $ 0.19425 \times 10^{-6} $\\ \\
$x^{16}-2 x^{15}-2 x^{14}-2 x^{13}-x^{11}-2 x^{10}-2 x^9-x^8-x^6-2 x^5-x^4-x-1 $ & $ 0.20095 \times 10^{-6} $\\ \\
$x^{16}-3 x^{15}+2 x^{13}-x^{12}-x^{11}+x^{10}+x^9-2 x^8+x^6-x^5+x^2-1 $ & $ 0.21102 \times 10^{-6} $\\ \\
$x^{16}-2 x^{15}-2 x^{14}-2 x^{13}-2 x^{12}-x^{11}-2 x^{10}-x^9+x^4-x-1 $ & $ 0.23696 \times 10^{-6} $\\ \\
$x^{15}-3 x^{14}+x^{11}-x^{10}+x^9-x^8+x^7+x^5-x^4-x^2+x-1 $ & $ 0.29620 \times 10^{-6}$\\
\bottomrule 
\end{tabular}\label{tab:fake}
\end{table}

In fact, this step resulted in a massive reduction of the list of candidates, with only $489$ Pisot numbers surviving: $271$ possible solutions to the equation $\al_1 + \al_2 = \al_3 + \al_4 $, $45$ possible solutions to the equation $\al_1 = \al_2 + \al_3 + \al_4$ and $173$ possible solutions to the equation $\al_1 + \al_2 + \al_3 + \al_4 = 0$.

In the last step, for each of the remaining $489$ minimal polynomials of Pisot numbers (including degree $4$ polynomials), the resultant polynomials $g(x) = \Res_y(f(x-y), f(y))$ and $h(x) = \Res_y(f(x+y), f(y))$ were calculated. The resultant polynomials were tested by checking the conditions described in Lemma \ref{resultant}. In particular, the Pisot polynomials that pass the resultant test given in part $(iii)$ of Lemma \ref{resultant} must have four distinct roots $\al_1,\al_2,\al_3,\al_4$ satisfying $\al_1+\al_2=\al_3+\al_4$. As a result, the only example that was found to pass this test was the original example given in \cite{smydub0}, namely, the polynomial
$
f(x)= x^4-2x^3+x-1.
$
No polynomial passed the tests $(iv)$-$(v)$ of Lemma \ref{resultant}. Thus, the equations $\al_1 = \al_2 + \al_3 + \al_4$ and $\al_1 + \al_2 + \al_3 + \al_4 = 0$ cannot be solved in conjugates of any Pisot number.
Therefore, the proof of Theorem \ref{viens} is completed.

Finally, we remark that all the post-processing was done with {\tt Maple} on the Mac Book air x86\_64 machine. In total, it took $55.72$ CPU hours.


\begin{thebibliography}{99}


\bibitem{bert}
{\sc M.J.~Bertin, A.~Decomps-Guilloux, M.~Grandet-Hugo, 
M.~Pathiaux-Delefosse and J.P.~Schreiber,}  \emph{Pisot and Salem numbers,} Birkh\"auser, Basel, 1992.

\bibitem{beu}
{\sc F.~Beukers and D.~Zagier,} \emph{Lower bounds of heights of points
on hypersurfaces,}
Acta Arith. {\bf 79} (1997), 103--111.

\bibitem{Boyd78}
{\sc D.~W. Boyd}, \emph{Pisot and Salem numbers in intervals of the real line,}
  Math. Comp. {\bf 32} (144) (1978), 1244--1260. 

\bibitem{Boyd84}
{\sc D.W.~Boyd,} \emph{Pisot numbers in the neighborhood of a limit point. {II},} Math.
  Comp. {\bf 43} (168) (1984), 593--602. 

\bibitem{Boyd85}
{\sc D.W.~Boyd}, \emph{Pisot numbers in the neighbourhood of a limit point. {I},} J.
  Number Theory {\bf 21} (1985), 17--43. 

\bibitem{dub}
{\sc A.~Dubickas,} {\it On the degree of  a linear form in conjugates of an
algebraic number,}
Illinois J. Math.  {\bf 46} (2002), 571--585.


\bibitem{DJ}
{\sc A.~Dubickas and J.~Jankauskas,} \emph{Simple linear relations with algebraic numbers of small degree,} submitted for publication, 2014.

\bibitem{smydub0}
{\sc A.~Dubickas and C.J.~Smyth,} \emph{On the lines passing through two conjugates of a Salem number,}
Math. Proc. Camb. Phil. Soc. {\bf 144} (2008), 29--37.



\bibitem{flint}
{\sc W.~Hart, F.~Johansson and S.~Pancratz,} \emph{{FLINT}: Fast Library for
  Number Theory,} 2013, Version 2.4.4, \url{http://flintlib.org}.

\bibitem{garza}
{\sc J.~Garza, M.I.M.~Ishak and C.~Pinner,} \emph{On the product of heights of algebraic numbers summing to real numbers,}
Acta Arith. {\bf 142} (2010), 51--58.

\bibitem{schi1}
{\sc G.~Hoehn and N.-P.~Skoruppa,} \emph{Un r\'esultat de Schinzel,}
J. Th\'eor. des Nombres de Bordeaux {\bf 5} (1993), 185.

\bibitem{kurb}
{\sc V.A.~Kurbatov,} \emph{Galois extensions of prime degree and their primitive elements,}
Soviet Math. (Izv. VUZ) {\bf 21} (1977), 49--52.

\bibitem{la1}
{\sc J.C.~Lagarias, H.A.~Porta and K.B.~Stolarsky,} \emph{Asymmetric tent map expansions I. Eventually
periodic points,} J. London Math. Soc. (2) {\bf 47} (1993), 542--556.

\bibitem{la2}
{\sc J.C.~Lagarias, H.A.~Porta and K.B.~Stolarsky,} 
\emph{Asymmetric tent map expansions II. Purely
periodic points,} Illinois J. Math. {\bf 38} (1994), 574--588. 

\bibitem{mig}
{\sc M.~Mignotte,} \emph{Sur les conjugu\'es des nombres de Pisot,}
C. R. Acad. Sci. Paris S\'er. I. Math. {\bf 298} (1984), 21.

\bibitem{nar} {\sc W.~Narkiewicz,} \emph{Elementary and analytic theory of algebraic numbers,} 3rd ed., Springer, Berlin, Heidelberg, 2004.

\bibitem{pad} {\sc R.~Padovan,} \emph{Dom Hans Van Der Laan and the plastic number}, Nexus IV: Architecture and Mathematics, Kim Williams Books, 2002, pp. 181--193.

\bibitem{salem}
{\sc R.~Salem,} \emph{A remarkable class of algebraic numbers. Proof of a conjecture of Vijayaraghavan,} Duke Math. J. {\bf 11} (1944), 103--108.

\bibitem{sam}
{\sc C.L.~Samuels,} \emph{Lower bounds on the projective heights of 
algebraic points,}
Acta Arith. {\bf 125} (2006), 41--50.

\bibitem{schi}
{\sc A.~Schinzel,} \emph{On the product of the conjugates outside the unit circle of an algebraic integer,}
Acta Arith. {\bf 24} (1973), 385--399.

\bibitem{sieg}
{\sc C.L.~Siegel,} \emph{Algebraic numbers whose conjugates lie in the unit circle}, Duke Math. J. {\bf 11} (1944),
597--602.

\bibitem{smy0}
{\sc C.~J.~Smyth,} \emph{On the product of the conjugates outside the unit circle of an algebraic integer,}
Bull. London Math. Soc. {\bf 3} (1971), 169--175.

\bibitem{smy1}
{\sc C.J.~Smyth,} \emph{The conjugates of algebraic integers,}
Amer. Math. Monthly {\bf 82} (1975), 86.

\bibitem{my}
 {\sc C.J.~Smyth,} \emph{Conjugate algebraic numbers on conics,} 
Acta Arith. {\bf 40} (1982), 333--346.

\bibitem{smy2}
{\sc C.J.~Smyth,} \emph{There are only eleven special Pisot numbers,}
Bull. London Math. Soc. {\bf 31} (1999), 1--5.

\bibitem{smy3}
{\sc C.J.~Smyth,} \emph{Mahler measure of one-variable polynomials: a survey.} Conference proceedings, University of Bristol, 3--7 April 2006, LMS Lecture Note Series 352, Cambridge University Press, Cambridge, 2008, pp. 322--349.

\bibitem{smy4}
{\sc C.J.~Smyth,} \emph{Seventy years of Salem numbers: a survey,} preprint, 2014,  \break arXiv:1408.0195v2.

\end{thebibliography}
\end{document}